\documentclass{amsart}
\usepackage{amsmath}
\usepackage{textcomp}
\usepackage{amsfonts}
\usepackage{amssymb,enumerate}
\usepackage{amsthm}
\usepackage{stmaryrd}
\usepackage[all]{xy}
\usepackage{hyperref}

\newcommand{\Reg}{\operatorname{Reg}}
\newcommand{\Spec}{\operatorname{Spec}}

\newcommand{\Ass}{\operatorname{Ass}}
\newcommand{\V}{\operatorname{V}}

\newcommand{\Z}{\operatorname{Z}}
\newcommand{\rr}{R\bowtie^f J}
\newcommand{\Nil}{\operatorname{Nil}}

\newcommand{\diam}{\operatorname {diam}}

\newcommand{\fm}{\frak{m}}
\newcommand{\fp}{\frak{p}}

\newcommand{\fq}{\frak{q}}

\newtheorem{thm}{Theorem}[section]
\newtheorem{cor}[thm]{Corollary}
\newtheorem{lem}[thm]{Lemma}
\newtheorem{prop}[thm]{Proposition}

\newtheorem{exam}[thm]{Example}
\newtheorem{rem}[thm]{Remark}

\begin{document}

\bibliographystyle{amsplain}

\date{}

\author{Y. Azimi}
\address{Department of Mathematics, University of Tabriz, Tabriz, Iran.}
\email{u.azimi@tabrizu.ac.ir}

\author{M.R. Doustimehr}
\address{Department of Mathematics, University of Tabriz, Tabriz, Iran;
	and School of Mathematics, Institute for Research in Fundamental Sciences (IPM), P.O. Box 19395-5746, Tehran, Iran.} 
\email{m\b{ }doustimehr@tabrizu.ac.ir}



\title[Zero-divisor graph of an amalgamated algebra]{The zero-divisor graph of an amalgamated algebra}

\begin{abstract}
	Let $R$ and $S$ be commutative rings with identity, $f:R\to S$ a ring homomorphism and $J$ an ideal of $S$. Then the subring $R\bowtie^fJ:=\{(r,f(r)+j)\mid r\in R$ and $j\in J\}$ of $R\times S$ is called the amalgamation of $R$ with $S$ along $J$ with respect to $f$.
	In this paper, we generalize and improve  recent results on the computation of the diameter of the zero-divisor graph of amalgamated algebras and obtain new results. In particular, we provide  new characterizations for completeness of  the zero-divisor graph of  amalgamated algebra, as well as, a complete description for the diameter of the zero-divisor graph of amalgamations in the special case of finite rings.
\end{abstract}

\maketitle

\section{Introduction}

%%%%%

The concept of the graph of the zero-divisors
of a ring $R$ was introduced by Beck in \cite{b}  and further studied by Anderson and Naseer in \cite{an}. In their work, all elements of the ring were vertices
of the graph. Then, in \cite{al}, Anderson and Livingston considered only the set of non-zero zero-divisors of $R$ as vertices. The definition
of Anderson and Livingston (which we adopt in this paper) is as follows.
The
\emph{zero-divisor graph} of $R$, denoted by $\Gamma (R)$, is the undirected graph
with vertices $\Z(R)^\bullet=\Z(R)\setminus \{0\}$, and for distinct
$x,y \in \Z(R)^\bullet$, the vertices $x$ and $y$ are adjacent if and only if $xy = 0$. 
Recall that the \emph{distance} between (connected) vertices $x$ and $y$, denoted by $d(x,y)$, is
the length of a shortest path connecting them.
The \emph{diameter} of a connected graph $G$, denoted by $\diam (G)$, is the supremum of the distances between vertices. By \cite[Theorem 2.3]{al}, $\Gamma (R)$ is connected and $\diam (\Gamma (R))\le 3$.
%%%%%%%%%%%%%%%%%%%%%

D'Anna, Finocchiaro, and Fontana in \cite{DFF} and \cite{DFF2} have
introduced the following construction. Let $R$ and $S$ be two
commutative rings with identity, let $J$ be an ideal of $S$ and let
$f:R\to S$ be a ring homomorphism. They introduced the following
subring (with standard component-wise operations)
$$R\bowtie^fJ:=\{(r,f(r)+j)\mid r\in R\text{ and }j\in J\}$$ of
$R\times S$, called the \emph{amalgamated algebra} (or \emph{amalgamation}) of $R$ with $S$ along $J$
with respect to $f$. This construction generalizes the amalgamated
duplication of a ring along an ideal (introduced and studied in
\cite{DF}). Moreover, several classical constructions such as Nagata's
idealization (cf. \cite[page 2]{Na}, the $R + XS[X]$ and
the $R+XS\llbracket X \rrbracket$ constructions can be studied as
particular cases of this new construction (see \cite[Example 2.5 and Remark 2.8]{DFF}). The construction has proved its worth 
providing numerous examples and counterexamples
in commutative ring theory.
%%%%%

Let $M$ be an $R$-module. In 1955, Nagata introduced
the \emph{idealization} of $M$ in $R$
(or the \emph{trivial extension} of $R$ by $M$),
denoted here by  $R\ltimes M$.
In \cite{as06}, Axtell and Stickles studied  the 
diameter and girth of the zero-divisor graph of 
a ring when extending to idealizations of the ring.

Recall that if $f:=id_R$ is the identity homomorphism
on $R$, and $I$ is an ideal of $R$, then $R\bowtie
I:=R\bowtie^{id_R} I$ is called the amalgamated duplication of $R$ along $I$.
In \cite{y}, Maimani and Yassemi investigated  
the diameter and girth of the zero-divisor graph of the ring $R\bowtie  I$.

Recently, in \cite{k}, Kabbaj and Mimouni
computed the diameter and girth of 
the zero-divisor graph for various contexts of  amalgamations.
 Then, in \cite{a},    previous results on the computation of the diameter of the zero-divisor graph of amalgamated algebras were
 improved and new results were obtained.

In this paper we pursue previous  works on the diameter 
of the zero-divisor graph of amalgamations.
Our attempt improves and recovers 
some of previous  results and 
yields new  results on amalgamated 
algebra, $n$-trivial extension (see definition
 in the
following) and
idealization.

The outline of the paper is as follows. In Section 2 
 we fix our notation and give some elementary results
 on which we base our approach.
In Section 3,  under certain assumptions, we
give equivalent conditions for completeness of 
 $\Gamma (\rr)$. 
Section 4  provides conditions under which 
the diameter of the zero-divisor graph of 
amalgamations is $2$ or $3$. 
In Section 5, we give a complete description of 
the diameter of the zero-divisor graph of 
amalgamations in the special case of finite rings.
Our results generalize  and complete 
well-known results mentioned above, and yield new and
original examples of graphs of
diameters $1$,  $2$, and $3$.

\section{Preliminaries}

To facilitate the reading of the paper, in this section, we recall  definitions
and prove some preliminary  properties we need later.

Let us first fix some notation which we shall use throughout the
paper: $R$ and $S$ are two commutative rings with identity, $J$ is a non-zero
proper ideal of the ring $S$, and $f:R\to S$ is a ring homomorphism. 
For a commutative ring $A$, the set of
zero-divisor elements, regular elements, nilpotent elements,
associate prime ideals,
and prime ideals of $A$ will be denoted by 
$\Z(A)$, $\Reg (A)$, $\Nil (A)$, $\Ass (A)$, and $\Spec (A)$, respectively. $\V(I)$ denotes the set of
prime ideals of $A$ containing $I$.
%$\Z(A)$ (resp., $\Reg(A)$) denotes the set of zero-divisors (resp., regular elements) of  $A$. 
For a subset $E$ of $A$, by $E^\bullet$ we 
mean $E\setminus \{0\}$.
%For abbreviation, we write ``iff" instead of 
%``if and only if".
Basic references for graph theory are \cite{di} and \cite{har}; for commutative ring theory, see \cite{Atiyah}, \cite{M}, and \cite{kap}.

\begin{rem}\label{zdd}
	\textnormal{The authors in \cite{ass} investigate the behavior of 
	the set of zero-divisors of amalgamated algebras as follows: Let $Z_1= \{(r,f(r)+j)\mid  r\in \Z(R)\}$
	and $Z_2= \{(r,f(r)+j)\mid j'(f(r)+j)=0$ for some $ j'\in J^\bullet \}$.
	Then $Z_2 \subseteq \Z(R\bowtie^f J) \subseteq Z_1\cup Z_2$. 
	The amalgamated ring $R\bowtie^f J$ is said to have 
	the \emph{condition $(\star)$} if the equality
	$\Z(R\bowtie^f J)= Z_1\cup Z_2$ holds.
	The amalgamated duplication of a ring along an ideal
	as well as idealization of a module in a ring has the condition $(\star)$.}
\end{rem}

In the following, we  provide a sufficient condition on $J$ so that 
$\rr$ has condition $(\star)$. However, before
we proceed in this direction, we are going to introduce, at this early stage, the  notion of
$n$-trivial extension introduced recently in
\cite{AAAA}. Let $M=(M_i)_{i=1}^n$ be a family of $R$-modules and
$\varphi=\{\varphi_{i,j}\}$ be a family of bilinear maps such that
each $\varphi_{i,j}$ is written multiplicatively:
\begin{align*}
\varphi_{i,j}:& M_i\times M_j  \longrightarrow M_{i+j} \\
&(m_i,m_j) \longmapsto  \varphi_{i,j}(m_i,m_j):=m_im_j
\end{align*}
The $n$-trivial extension of $R$ by $M$ is the set denoted by
$R\ltimes_nM$ whose underlying additive group is $R\oplus
M_1\oplus\cdots\oplus M_n$ with multiplication given by
$$(m_0,\ldots,m_n)(m'_0,\ldots,m'_n)=\Big(\sum_{j+k=i} m_jm'_k\Big)_i$$
for all $(m_0,\ldots,m_n),(m'_0,\ldots,m'_n)\in R\ltimes_nM$. It is
shown that $R\ltimes_nM$ is associative ring precisely when
$(m_im_j)m_k=m_i(m_jm_k)$ for $m_i\in M_i, m_j\in M_j$ and $m_k\in
M_k$ with $1\leq i, j, k \leq n - 2$ and $i + j + k\leq n$. Further,
for $R\ltimes_nM$ to be a commutative ring with identity we need $R$
to be commutative with identity and $m_im_j=m_jm_i$ for every $1\leq
i, j \leq n-1$ with $i+j\leq n$. It can be seen that $(0\ltimes_n
M)^{n+1}=0$. It is also easy to show that the $n$-trivial extension
is a special case of the amalgamated construction. Indeed, let
$S:=R\ltimes_nM$, $J:=0\ltimes_nM$ and $f:R\to S$ be the natural
injection. It is now clear that $\varphi:R\bowtie^f J \to
R\ltimes_nM$, $(r,f(r)+(0,m_1,\ldots,m_n))\mapsto(r,m_1,\ldots,m_n)$
is an isomorphism.

\begin{lem}\label{star}
	 Let $J\subseteq \Nil (S)$.
		Then the amalgamated ring $R\bowtie^f J$ has
		the condition $(\star)$. In particular, the $n$-trivial extension of a ring by a family of $R$-modules
		 has the condition $(\star)$. 
\end{lem}
\begin{proof}
Let $r\in \Z (R)^\bullet$ and $i\in J$. Then, for some 
$s \in R$,  $rs=0$. If $f(s)J=0$, then
$(r,f(r)+i)(s,f(s))=0$, as desired.
 Next assume that,
for some $j \in J$, $f(s)j\neq 0$.
If $i = 0$, then $(r, f (r))(s, f (s)) =0$. 
Assume that $i \neq 0$ and let $n$ be the smallest positive integer such that $i^n= 0$ and
$m \leq n-1$ the largest integer such that $f (s) ji^m \neq 0$.
 Then
 $(r,f(r)+i)(0,f(s)ji^m)=0$, and the desired
 result follows.
\end{proof}

In the next remark we recall the
structure of the prime spectrum of the ring $\rr$ and 
 use it in this paper without comments.
\begin{rem}\label{spec}
	\textnormal{(\cite[Corollary 2.5]{DFF1})
	For $\fp\in\Spec(R)$ and $\fq\in\Spec(S)\setminus \V(J)$, set
	\begin{align*}
	\fp^{\prime_f}:= & \fp\bowtie^fJ:=\{(p,f(p)+j)|p\in \fp, j\in J\}, \\[1ex]
	\overline{\fq}^f:= & \{(r,f(r)+j)|r\in R, j\in J, f(r)+j\in \fq\}.
	\end{align*}
	Then, the prime ideals of $R\bowtie^fJ$ are of the type $\overline{\fq}^f$ or $\fp^{\prime_f}$, for
	$\fq$ varying in $\Spec(S)\setminus \V(J)$ and $\fp$ in $\Spec(R)$.
	Thus $\rr$ does not have any 
	prime ideal of the form $\overline{\fq}^f$ if and only if
	$J\subseteq \Nil (S)$.}
\end{rem}

\section{When is the graph $\Gamma(R\bowtie^fJ)$ complete?}
In \cite[Theorem 2.1]{k}, Kabbaj and Mimouni characterize 
completeness of the zero-divisor graph
 of amalgamations under the assumption
that $f^{-1}(J)\neq 0$. This assumption is reasonable and
includes amalgamated duplications, 
but  the theorem exclude
idealizations as a special case of amalgamated algebras. The theorem also has the 
advantage of not assuming  $\Gamma (R)\neq \emptyset$ (i.e. $R$ can be a domain).  
In \cite[Theorem 3.3]{as06}, Axtell and Stickles characterize 
 completeness of a zero-divisor graph of  idealization of 
 a module $M$ in $R$, when $\Gamma (R)\neq \emptyset$.
 To this end, they state the following three properties:
\begin{itemize}
\item [(A)] $\Z(R)^2=0$.
\item [(B)] For every $r\in \Reg (R)$ and 
every $m\in M^\bullet$, $rm\neq 0$.	
\item [(C)] For every $r\in \Z (R)^\bullet$, $rM=0$.
\end{itemize} 
In this section, we give a characterization of 
completeness of the zero-divisor graph of amalgamations
with assumptions that include both amalgamated duplication and
idealization. This characterization
 generalizes \cite[Theorem 3.3]{as06}
and \cite[Theorem 3.8]{y}.
We first need to propose suitable version of properties (A), (B), and (C):
\begin{itemize}
	\item [(a)] $\Z(R)^2=0$.
	\item [(b)] For every $r\in \Reg (R)$ and every $j\in J^\bullet$,
	 $jf(r)\neq 0$.	
	\item [(c)] For every $r\in \Z (R)^\bullet$ and every $j\in J^\bullet$,
	$jf(r)= 0$.
	\item [(d)]	$J^2=0$.
\end{itemize} 
We use the above notation throughout this paper.

\begin{thm} \label{cm}
Let $\Gamma (R)\neq \emptyset$ and
suppose that $R\bowtie^f J$ 
has the condition $(\star)$. Then the following
 statements are equivalent: 
\begin{itemize}
\item [(1)]
	$\Gamma (R\bowtie^f J)$ is complete;
\item [(2)]
	Properties (a), (b), (c), and (d) hold;	
\item [(3)]
   $\Z (\rr)^2=0$.
\end{itemize}
In fact, the implications
$(2)\Rightarrow (3) \Rightarrow (1)$ hold without any 
assumption. 

\end{thm}
\begin{proof}
$(1)\Rightarrow (2)$ Assume that $\Gamma (R\bowtie^f J)$ is complete.\\
(a) Let $r,s\in Z(R)^\bullet$ and pick an element $j\in J^\bullet$.
By Remark \ref{zdd}, $(r,f(r))$ and $(s,f(s)+j)$ are
(different) elements of  $\Z(\rr)$.
Thus $(r,f(r)) (s,f(s)+j)=0$ which implies that $rs=0$.
Hence $\Z(R)^2=0$.\\
(b)  Let $r\in \Reg (R)$ and $j\in J^\bullet$ and
suppose, on the contrary, that $jf(r)= 0$.	
Then $(r,f(r))\in \Z (\rr)$. Let $s\in \Z(R)^\bullet$.
%(and so $(s,f(s))\in \Z (\rr)$).
 By assumption, 
$(r,f(r))(s,f(s))=0$ and so $rs=0$, a contradiction.\\
(c) Let $r\in \Z (R)^\bullet$  and $j\in J^\bullet$.
Then $(r,f(r))$ and $(0,j)$ are different elements
of $\Z (\rr)$ and so, by assumption,
$(r,f(r)) (0,j)=0$.
Thus $jf(r)=0$.\\
(d) Let $i$ and $j$ be (possibly equal) elements of
$J^\bullet$ and pick $r\in \Z(R)^\bullet$.
It follows from Remark \ref{zdd} that 
$(r,f(r))$, $(r,f(r)+j)$, and $(0,i)$
are different elements of $\Z (\rr)$.
Therefore $(r,f(r))(0,i)=0=(r,f(r)+j)(0,i)$ and so $ij=0$.
 
$(2)\Rightarrow (3)$ Assume that properties
(a), (b), (c) and (d) hold.
Let $(r,f(r)+i), (s,f(s)+j)\in \Z (\rr)$.
Properties (b) and (d) together with Remark \ref{zdd} yield 
$r,s\in \Z(R)$.
Then, using (a), (c) and (d) we have
$(r,f(r)+i)(s,f(s)+j)=0$.
% $rs=0$ and using (c) we have 
%$jf(r)=i(f(s)+j)=0$. Thus $(r,f(r)+i)(s,f(s)+j)=0$.

$(3)\Rightarrow (1)$ This is trivial.
\end{proof}

\begin{cor}\label{y 3.8} 
	(\cite[Theorem 4.8]{y}).
	Let I be a non-zero ideal of R. Then the following are equivalent:
	\begin{itemize}
		\item [(1)]
		The graph $\Gamma(R\bowtie I)$ is a complete graph;
		\item [(2)]
		$\Z(R)^2=0$ and $I\subseteq \Z(R)$;	
		\item [(3)]
		$\Z (R\bowtie I)^2=0$.
	\end{itemize}
\end{cor}
\begin{proof}
	It is easy to see that if
	$\Gamma(R\bowtie I)$ is  complete, then $R$ is
	not a domain (see \cite[Theorem 3.17]{a}). 
	Thus assumptions of Theorem \ref{cm} hold.
\end{proof}

\begin{prop} \label{cm3}
	Let $J^2=0$. Then, the following
	conditions are equivalent: 
	\begin{itemize}
		\item [(1)]
		$\Gamma (R\bowtie^f J)$ is complete;
		\item [(2)]
		Properties (a), (b), and (c) hold;	
		\item [(3)]
		$\Z (\rr)^2=0$.
	\end{itemize}
\end{prop}
\begin{proof}
	As in the
	proof of Theorem \ref{cm}, we only need to prove
	that if $\Gamma (R\bowtie^f J)$ is complete, then (b) holds. Let
	$\Gamma (R\bowtie^f J)$ be complete, 
	$r\in \Reg (R)$, $j\in J^\bullet$, and
	suppose on the contrary that $jf(r)= 0$.	
	Then $(r,f(r)), (r,f(r)+j)\in \Z (\rr)$. Thus $r^2=0$, a contradiction.
\end{proof}

As in
\cite[Remark 2.8]{DFF}, if $S:=R\ltimes M$, $J:=0\ltimes M$, and
$\iota:R\to S$ be the natural embedding, then $R\bowtie^\iota J\cong
R\ltimes M$ which maps the element $(r,\iota(r)+(0,m))$ to the
element $(r,m)$.
Thus idealization is 
a special case of amalgamation with $J^2=0$.
Therefore, the above corollary generalizes
and improves \cite[Theorem 3.3]{as06}, dropping the 
assumption $\Gamma (R)\neq \emptyset$.

\begin{cor}\label{as3.6} 
	(c.f \cite[Theorem 3.3]{as06}).
	Let $M$ be an $R$-module. Then
	$\Gamma (R\ltimes M)$ is complete if and only if
	properties (A), (B), and (C) hold.
\end{cor}
%%%%%%%%%%%%%%%

In the above theorem, we
assumed that $R$ is  not a  domain. In fact, if 
$R$ is a  domain, then the resulting statement is not
always true. This is illustrated by the following example.
\begin{exam}
	Let $R=\mathbb{Z}_2$, $S=\mathbb{Z}_{2}\times \mathbb{Z}_{2}$, and
	$J$ be the ideal generated by $(0,1)$. Let $f:R\to S$
	be the natural homomorphism. Then
	$\rr \cong \mathbb{Z}_{2}\times \mathbb{Z}_{2}$ 
	and so $\Gamma (\rr)$ is complete, while $J^2\neq 0$ 
	and $\Z (\rr)^2\neq 0$.
\end{exam}

The next example 
shows
that, in Theorem \ref{cm}, the hypothesis that $\rr$ satisfies condition ($\star$)  is necessary.
\begin{exam}\label{z4}
	Let $R=\mathbb{Z}_4$, $S=\mathbb{Z}_{4}[X]$,
	$J= X\mathbb{Z}_{4}[X]$, and $f:R\to S$ be the natural inclusion
	map. We have $f^{-1}(J)=0$,
	and so $\rr \cong f(R)+J=\mathbb{Z}_{4}[X]$. 
	Moreover $\rr$ dose not
	have the condition $(\star)$, 
	since $2\in \Z(R)$ but $(2,2+X)\notin \Z(\rr)$, 
	by Remark \ref{zdd}.\\
	Note that $\Z(\mathbb{Z}_{4}[X])= 2\mathbb{Z}_{4}[X]$ hence that $\Gamma (\rr)$ is complete.	
	It is easy to see that	$\rr$
	satisfies properties (a) and (b) but not 
	(c) and (d).
\end{exam}

\begin{exam}
	Let $R=\mathbb{Z}_2$,   $S=\mathbb{Z}_{2}[X]/(X^3)$,  and   $f:R\to S$ be the composition of natural 
	maps $\mathbb{Z}_{2} \to \mathbb{Z}_{2}[X] \to \mathbb{Z}_{2}[X]/(X^3)$.
	Let $x$ denote the class of $X$ modulo $(X^3)$ and consider the following cases: 
  \begin{itemize}
  	\item [(1)] $J= (x)$  and,	
  	\item [(2)] $J= (x)^2$.
  \end{itemize}  
In either case, it is clear that 
(a), (b), and (c) hold. Note that $f^{-1}(J)=0$,
% hence that $\rr \cong f(R)+J=\mathbb{Z}_{2}[X]/(X^3)$. 
hence that $\rr \cong f(R)+J=\mathbb{Z}_{2}[x]$.
Note also that 
$\rr$ has condition $(\star)$ and $\Gamma (R)= \emptyset$.

In the first case $J^2\neq 0$. Thus we
do not have conditions of Theorem \ref{cm},
nor do we  have conditions of Proposition \ref{cm3}.
It is easy to see that $\Z(\rr)^2\neq 0$ and 
$\Gamma (R\bowtie^f J)$ is not complete, as desired.\\
In the second case $J^2=0$. Thus by Proposition \ref{cm3},
$\Z(\rr)^2=0$ and
$\Gamma (R\bowtie^f J)$ is complete.
\end{exam}

We will prove the following characterization for
completeness of the zero-divisor graph of amalgamations
 in part (1) of Theorem \ref{finite}.
\begin{prop} \label{cmf}
Let $R\bowtie^f J$ have finitely many  elements
 and $\rr \ncong \mathbb{Z}_{2}\times \mathbb{Z}_{2}$.
 Then the following statements are equivalent: 
	\begin{itemize}
		\item [(1)]
		$\Gamma (R\bowtie^f J)$ is complete;
		\item [(2)]
		Properties (a), (b), (c), and (d) hold;	
		\item [(3)]
		$\Z (\rr)^2=0$.
	\end{itemize}
\end{prop}

%%%%%%%%%%%%%%%%%%%%%%%%%%%%%%%%%%%%%%%5
\section{When is the graph $\Gamma(R\bowtie^fJ)$ of diameter $2$ or $3$?}

In \cite{as06}, Axtell and Stickles
 prove that if $R\ltimes M$ satisfies properties
(A), (B) but not (C), or
if $R\ltimes M$ satisfies properties
(B), (C) but not (A), then $\diam  (\Gamma(R\ltimes M))=2$
(\cite[Theorem 3.9]{as06} and \cite[Theorem 3.10]{as06}). 
In this section, we investigate when similar
 results hold for $R\bowtie^f J$ based on
properties (a), (b), and (c). Note that property 
(d) automatically holds whenever $\rr$ is an idealization.

Note that in either of above cases
$\Z(R)$ is  an ideal of $R$.
In \cite[Theorem 3.10]{k} and \cite[Theorem 3.1]{a},
under certain conditions including property (b),
the authors prove  that if
$\Z(R)$ is \emph{not} an ideal of $R$, then 
$\diam (\Gamma(\rr))=3$. It is natural to ask, whether
$\diam (\Gamma(\rr))=2$  provided that 
$\Z(R)$ is an ideal of $R$. The results
of this section provide a partial answer to this question.
Theorem \ref{id 3.9} and Theorem \ref{id 3.10}
provide conditions under which $\diam (\Gamma(\rr))=2$.
But Proposition \ref{ex} shows that it also may 
happen that $\diam (\Gamma(\rr))=3$ when 
$\Z(R)$ is  ideal.

%%%%%%%%%%%%%%%%%%%%%%%%%%%%%%%%%%%%%%%%%%%%%%%

\begin{thm}\label{id 3.9}
	Let $J\subseteq \Nil (S)$.
	If $\rr$ satisfies properties  (b)
	and (c) but not property (a),
	then  $\diam  (\Gamma(R\bowtie^f J))=2$.
\end{thm} 
\begin{proof}
	We first observe, by
	% (the proof of)
	  Theorem \ref{cm},  that
	$\diam  (\Gamma(R\bowtie^f J))\ge 2$, since 
	property (a) does not hold. Consider non-adjacent vertices
	$(r,f(r)+i)$ and $(s,f(s)+j)$ in $\Gamma(\rr)$.
	We claim that  $r,s\in \Z(R)$.
	Indeed, for some $\mathcal{L}\subseteq \Spec(R)$,	
	$(r,f(r)+i) \in \Z(\rr)=\cup_{\mathcal{P}\in \mathcal{L}} \mathcal{P}$, 
	and $(0,i)\in \Nil (R\bowtie^f J)=\cap_{\mathcal{P}\in \Spec (\rr)}\mathcal{P}$.
	Thus $(r,f(r)) \in \cup_{\mathcal{P}\in \mathcal{L}} \mathcal{P}=\Z(\rr)$, and so, by the property (b) and Remark \ref{zdd},
	$r\in \Z(R)$.
	If we prove that there exists $k\in J^\bullet$, such that
	 $ki=kj=0$, then we would have the path 
	 $(r,f(r)+i)-(0,k)-(s,f(s)+j)$ of length $2$ and
	 the assertion follows.
	
	First note that, if $ij=0$, then we  take $k$
	to be a suitable power of $i$. Next,
	since $J\subseteq \Nil (S)$, there exists
	a smallest $m\in \Bbb N$ such that $i^mj^m=0$.
	Choose  $n\in \Bbb N$
	such that $i^{m-1}j^n=0$ and $i^{m-1}j^{n-1}\neq 0$.
	Three cases arise:\\ \textbf{Case 1.} $n> m$. Then
	we have the path $i-i^{m-1}j^{n-1}-j$.\\ 
	\textbf{Case 2.} $n= m$ and $i^{m}j^{n-1}= 0$. Again
	we have the path in the previous case.\\ \textbf{Case 3.} $n= m$ and $i^{m}j^{n-1}\neq 0$.
	Take a $p\in \Bbb N$ 
	such that $i^pj^{m-1}=0$ and $i^{p-1}j^{m-1}\neq 0$. 
	Then $p>m$ and we have the path $i-i^{p-1}j^{m-1}-j$.
\end{proof}

%%%%%%%%%%%%%%%%%%%%%%%%%%5
\begin{cor}
	\begin{itemize}
		\item [(1)] Let $M=(M_i)_{i=1}^n$ be a family of $R$-modules. Assume that $R$ and each  $M_i$, 
		$1\le i \le n$, satisfy properties  (B)
		and (C) but not property (A).
		Then  $\diam  (\Gamma(R\ltimes_n M))=2$.		
		\item [(2)] \cite[Theorem 3.9]{as06}.
		Let $M$ be an $R$-module. 
		If $R\ltimes M$ satisfies properties  (B)
		and (C) but not property (A),
		then  $\diam  (\Gamma(R\ltimes M))=2$.
	\end{itemize}
\end{cor}

\begin{exam}
	Let $R=\mathbb{Z}[X]/(X^6)$, $S=\mathbb{Z}[X]/(X^2)$,
	$J=(X/X^2)$,  and   $f:R\to S$ be the   natural 
	homomorphism. Then  $\diam  (\Gamma(R\bowtie^f J))=2$,
	by Theorem \ref{id 3.9}.
\end{exam}

The
assumption $J\subseteq \Nil (S)$ in the above
theorem  is not superfluous.    
We need the following proposition 
to construct an example that shows this.
Recall that an ideal of a ring is called a 
regular ideal if it contains a regular element.

%%%%%%%%%%%%%%%%%%%%%%%%%%%%%%%%%%%%%%%%%%%%%%%%
\begin{prop} \label{ex}
	Assume that  $\Z (R)$ is an ideal of $R$, and $J$
	and $f^{-1} (J)$ are regular ideals. 
	If for some non-zero $k\in J$ and some non-zero $f(x)+i\in f(R)+J$, we have $k(f(x)+i)= 0$,
	then $\diam  (R\bowtie^f J)=3$.
\end{prop}
\begin{proof}
	Let $k$ and $f(x)+i$ be as assumption and $j\in J$
	be regular.
	Consider the following two cases.\\
	\textbf{Case 1.} $x\in \Reg (R)$. Then $(x,f(x)+i)$ is adjacent to $(0,k)$ and this type of 
	zero-divisors are the only ones that
	 $(x,f(x)+i)$ can be adjacent to them. 
	 On the other hand since $j\in \Reg (S)$, $(0,j)$ 
	 can be adjacent only to zero-divisors of 
	 type $(a,0)$. Let $a\in f^{-1}(J)\cap \Reg (R)$. 
	 Then one can not shorten  the following path:
	 $$(x,f(x)+i) - (0,k) - (a,0) -  (0,j).$$	
	\textbf{Case 2.} $x\in \Z (R)$.  Let $a\in f^{-1}(J)\cap \Reg (R)$. By assumption, 
	 $x+a\in \Reg (R)$.  Now consider the element 
	 $(x+a,f(x)+i)$, which is adjacent to $(0,k)$. The remaining of the proof is just as the case 1.
\end{proof}

Here is an example that shows how one can build graphs
 of diameter 3 using the above proposition. This
  example also shows that the
 assumption $J\subseteq \Nil (S)$ in 
Theorem \ref{id 3.9} is essential.

\begin{exam}\label{exam}
		Let $R= k[X,Y]/(X^3)$, $S= k[Y,Z]/(Z^2)$,
		$f: R \to S$
		be the composition of natural homomorphisms
		$k[X,Y]/(X^3)\to k[Y] \to k[Y,Z] \to k[Y,Z]/(Z^2)$,
		and  $J=(Y,Z)S$. Let
		$x, y, z$ denote congruent classes of $X, Y, Z$.
		Then $J$ and $f^{-1}(J)$ are regular ideals, and 
		 $\Z (R)$ is an ideal. Moreover, we have
		 $z(f(x)+z)=0$.		 
		Thus $\diam (\Gamma (R\bowtie^f J))=3$ by Proposition \ref{ex}.\\			
		Notice that $\rr$ satisfies properties (b)
		and (c) but not property (a).
		Note also that $J\nsubseteq \Nil (S)$.
\end{exam}
%%%%%%%%%%%%%%%%%%%%%%%%%%%%%%%%%%%%%%

\begin{thm}\label{id 3.10}
	Let $J\subseteq \Nil (S)$.
	If $\rr$ satisfies properties  (b)
	and (a) but not property (c),
	then  $\diam  (\Gamma(R\bowtie^f J))=2$.
\end{thm} 
\begin{proof}
	Since  property (c) does not hold, 
	it follows that there exists $x\in \Z(R)^\bullet$
	and $k\in J^\bullet$ 
	  such that $f(x)k\neq 0$.
	Then
	$(0,k), (x,f(x))\in \Z(\rr)^\bullet$, and  
	$(x,f(x))(0,k)=(0,f(x)k)\neq 0$, which implies 
	$\diam  (\Gamma(R\bowtie^f J))\ge 2$. To
	prove the converse inequality, note that
	as in the proof of Theorem \ref{id 3.9}, for any $(r,f(r)+i), (s,f(s)+j)\in \Z(\rr)$
	one has $r,s\in \Z(R)$. Thus, using property (a),
	 we have $f(r)f(x)k=f(s)f(x)k=0$.
	If, in addition, $f(x)ki=f(x)kj=0$, we
	obtain the path $(r,f(r)+i)-(0,f(x)k) - (s,f(s)+j)$, as desired.
	If not, say $f(x)ki\neq 0$, we replace $f(x)k$ with
	$f(x)ki$. Repeating this procedure
	 (which terminate since $J\subseteq \Nil (S)$) we
	 obtain the path $(r,f(r)+i)-(0,f(x)ki^mj^n) - (s,f(s)+j)$
	 of length $2$, for suitable
	  $m,n\in \Bbb N_0$.
	Therefore $d((r,f(r)+i),(s,f(s)+j))\le 2$; thus
	 $\diam  (\Gamma(R\bowtie^f J))\le 2$.
\end{proof}
%%%%%%%%%%%%%
\begin{cor}\label{47}
	\begin{itemize}
		\item [(1)] Let $M=(M_i)_{i=1}^n$ be a family of $R$-modules. Assume that $R$ and each  $M_i$, 
		$1\le i \le n$, satisfy properties  (B)
		and (A) but not property (C).
		then  $\diam  (\Gamma(R\ltimes_n M))=2$.	
		(In fact we need \emph{only one} $M_i$ not to satisfy  property (C), which is much more available).	
		\item [(2)] \cite[Theorem 3.10]{as06}.
		Let $M$ be an $R$-module. 
		If $R\ltimes M$ satisfies properties  (B)
		and (A) but not property (C),
		then  $\diam  (\Gamma(R\ltimes M))=2$.
	\end{itemize}
\end{cor}

%The next corollary helps us to construct  a graph
%of diameter $2$  using any complete zero-divisor graph 
%other than $\Gamma(\mathbb{Z}_{2}\times \mathbb{Z}_{2})$.
\begin{cor}
	Let $R$ be a ring with $\Z(R)^2=0$. Then  
	$\diam (\Gamma (R[X]/(X^n)))=2$.
\end{cor}
\begin{proof}
	Let $S=R[X]/(X^n)$, $J= (X/X^n)$, and $f: R \to S$
	be the natural homomorphism. Then $f^{-1}(J)=0$,
	and so $\rr \cong f(R)+J=R[X]/(X^n)$. 
	By Theorem \ref{id 3.10}, 
	$\diam (\Gamma (R[X]/(X^n)))=2$.
\end{proof}
\begin{exam}
	Let $R=\mathbb{Z}_4$, $S=\mathbb{Z}_{4}[X]/(X^3)$,
	$J=(X/X^3)$, and $f:R\to S$ be the natural homomorphism.
	Then $\diam (\Gamma (\rr))=2$.
\end{exam}

Example \ref{z4} shows that in the above theorem
the  assumption $J\subseteq \Nil (S)$ is  crucial. In that 
case we have $J\nsubseteq \Nil (S)$ and $\diam  (\Gamma(R\bowtie^f J))=1$.
The next example provides conditions under which the hypothesis $J\subseteq \Nil (S)$ fails and $\diam  (\Gamma(R\bowtie^f J))=3$.

\begin{exam}
	Let $R$ be a ring which is not a domain and $\Z(R)^2=0$.
	If $I$ is regular ideal of $R$,
	then by \cite[Corollary 3.2]{a} $\diam  (\Gamma(R\bowtie I))=3$.
	Note that $R\bowtie I$
	satisfies  properties (a)
	and (b) but not  (c) and (d).
	To construct a concrete example, one can
	 consider $R=\Bbb Z[X]/(X^2)$, and $I=(\bar 2,x)R$,
	where $\bar 2$ and $x$ denote the classes of $2$ and $X$ modulo $(X^2)$.
\end{exam}

%%%%%%%%%%%%%%%%%%%%%%%%%%%%%%%%%%
%%%%%%%%%%%%%%%%%%%%%%%%%%%%%%%%%%%%%%%%%5
\section{Finite ring case}

In \cite[Theorem 2.2]{al}, Anderson and Livingston state
that $\Gamma (R)$
is finite if and
only if either $R$ is finite or an integral domain.
Thus $\Gamma (R)$  may be infinite. But it
is probably of most interest when it is finite, for
then one can draw it.

In this section, we deal with the case that
 the ring is  finite  (i.e has
finitely many elements).
Assume that $\Z(R)=\cup_{\fp\in \Ass (R)} \fp$ is an ideal.
Then $\Z(R)$ is a prime ideal and, since $R$ is finite,
it is the
unique maximal ideal of $R$. Thus the set of zero-divisors of
a finite ring $R$ is an ideal
if and only if $R$ has only one prime ideal.

To shorten and ease reading of 
the  proof of our last main result
we state the following elementary lemmas.
\begin{lem}\label{3}
	Assume that $\rr$ is a finite ring. Then
	$\Z (\rr)$ is an ideal if and only if
	$\Z (R)$ is  an ideal and $J\subseteq \Nil (S)$.
\end{lem}
\begin{proof}
    $(\Longrightarrow)$ If $\Z (R)$ is not an ideal, then $R$ has at least
	two prime ideals and so does $\rr$. Thus 
	$\Z (\rr)$ is not an ideal.
	If $J\nsubseteq \Nil (S)$, then $\rr$ 
	has at least one prime ideal of the 
	form $\overline{\fq}^f$, by Remark \ref{spec}.
	Note that $\rr$ always
	has at least one prime ideal of the form $\fp^{\prime_f}$
	and $\fp^{\prime_f}\neq \overline{\fq}^f$.
	Thus $\rr$ has at least
	two prime ideals and so $\Z (\rr)$ is not an ideal.
	
	$(\Longleftarrow)$ By Remark \ref{spec}, $\rr$ do not have any 
	prime ideal of the form $\overline{\fq}^f$, since
	$J\subseteq \Nil (S)$.
	Also $\rr$ has only one prime ideal of
	the form $\fp^{\prime_f}$ Since, by assumption, $R$ has 
	only one prime ideal. Thus $\Z (\rr)$ is an ideal.	
\end{proof}

The following lemma, which is of independent
interest, has important role in the proof of
the main result of this section. The proof of
 this lemma
 is similar to 
the final part of the proof of Theorem \ref{id 3.9}
 and so it is omitted. Notice that 
 for a finite ring $R$, if $\Z (R)$ is
  an ideal, then $\Z (R)=\Nil (R)$.
\begin{lem}\label{finite}
	Let $R$ be  a finite ring and
	$\Z (R)$ be an ideal of $R$.
	Then $\diam \Gamma (R) \leq 2$.
\end{lem}

 We will use the following
lemma without explicit mention.
%We use the following lemma in the proof of the theorem.
\begin{lem}\label{min} 
	(\cite[Lemma 2.3]{a}) The following statements hold.
	\begin{itemize}
		\item [(1)] Let $\iota:R\to R\bowtie^f J$,
		defined by
		$\iota  (x)=(x,f(x))$ for every $x\in R$,
		be integral (e.g. $|\rr|<\infty$).
		If $\fp$ is a minimal prime of $R$, then $\fp^{\prime_f}$ is a minimal prime of $R\bowtie^f J$.
		\item [(2)] If a prime ideal 
		$\fq$ of $S$ is minimal  with the property $J \nsubseteq \fq$, then $\overline{\fq}^f$ is a minimal prime of $R\bowtie^f J$.
	\end{itemize}
\end{lem}

By \cite[Theorem 2.3]{al}, for any commutative ring $A$,
$0\le \diam (\Gamma (A))\le 3$.
The main result of this section characterizes the diameter of zero-divisor graph of  amalgamated algebra  $\Gamma (R\bowtie^f J)$ 
when $\rr$ (or equivalently both $R$ and $J$) has finitely many  elements.

\begin{thm}\label{fine}
	Let $R\bowtie^f J$ 
	have finitely many  elements. Then the following hold:
	\begin{itemize}
		\item [(1)]
		$\diam (\Gamma (\rr))=0$ if and only if $R$ is a field and 
		 $|J^\bullet|=1$ with $J^2=0$.
		 Note that, when this is the case, 
		 properties (a), (b), (c), and (d) hold.
		\item [(2)]
		$\diam (\Gamma (\rr))=1$ if and only if
		$\rr \cong \mathbb{Z}_{2}\times \mathbb{Z}_{2}$, or
		properties (a), (b), (c), and (d) hold and 
		$\diam (\Gamma (\rr))>0$.	
		\item [(3)]
		$\diam (\Gamma (\rr))=2$ if and only if $\diam (\Gamma (\rr))>1$ 
		and either (i) $\Z (R)$ 
		is an ideal and $J\subseteq \Nil (S)$, or (ii)
		$J\cap \Nil (S) = \{0\}$,
		$S$ has exactly one prime ideal $\fq$, 
		 minimal  with the property $J \nsubseteq \fq$, and
		 $R$ is a domain.
		\item [(4)]
		$\diam (\Gamma (\rr))=3$ if and only if the following hold:
		(i) $\Z (R)$ is not an ideal or $J\nsubseteq \Nil (S)$, and
		(ii) $J\cap \Nil (S) \neq \{0\}$, or 
		 $S$ has more than one  prime ideal $\fq$, 
		 minimal  with the property $J \nsubseteq \fq$, or
		 $R$ is not a domain.
	\end{itemize}
\end{thm}
\begin{proof}
	(1) Assume that $\diam (\Gamma (\rr))=0$, i.e. $|\Z (\rr)^\bullet|=1$.
	Then, by the comments in the beginning of 
	this section, $\rr$ has exactly one prime
	 ideal, $\fm^{\prime_f}$, hence
	 $J\subseteq \Nil (S)$. 
	 When this is the case, every element of $\fm$ and $J$
	 gives a zero-divisor in $\rr$, by Remark \ref{zdd} and
	 Lemma \ref{star}. 
	 It follows easily that
	 %$\fm^{\prime_f}=0$, and $|J^\bullet|=1$ with $J^2=0$.
	$\fm =0$, and $|J^\bullet|=1$ with $J^2=0$ (if
	$J=\{0,j\}$ and $j^2\neq 0$, then $j^2=j$ and so $j^n\neq 0$, for any $n\in \mathbb N$).
	The  converse statement is clear.
	
	(2) Let $\diam (\Gamma (\rr))=1$ and
	$\rr \ncong \mathbb{Z}_{2}\times \mathbb{Z}_{2}$.
	Then, by \cite[Theorem 2.8]{al}, $\Z(\rr)^2=0$, and so
	$\Z(\rr)$ is an ideal. Thus $\Z(\rr)$ is the only
	prime ideal of $\rr$, call it $\fm^{\prime_f}$.
	Hence $\fm ^2=0$ and $J^2=0$.
	Proposition \ref{cm3} now completes the proof.

	We now prove ``if" direction of parts (3) and (4).
	Then, ``only if" direction follows, since
	conditions of parts  (1)-(4) are logically
	disjoint. To see this, note that conditions of part (4)
	are logical opposite of ones of part (3).
	Note also that if part (4) holds, then either
	$\Z (R)$ is not an ideal, which implies property (a) not
	holds, or $J\nsubseteq \Nil (S)$, and so property (d) not
	holds. One can easily deduce from Remark \ref{5} that
	if part (4) holds, then 
	$\rr \ncong \mathbb{Z}_{2}\times \mathbb{Z}_{2}$
	
	(3)  Assume that $\diam (\Gamma (\rr))\geq 2$.
	We prove the converse inequality in each case:\\	 
	(i) Let $\Z (R)$ 
	be an ideal and $J\subseteq \Nil (S)$.
	In view of Lemma \ref{3}, $\Z (\rr)$ is an ideal, 
	and so, by Lemma \ref{finite},
	$\diam (\Gamma (\rr))\leq 2$.\\
	(ii) By Lemma \ref{3}, $\Z (\rr)$ is not an ideal,
	and by \cite[Proposition 5.4]{DFF}, $\rr$ is reduced.
	Let $\fq$ be the only  prime ideal of 
	$S$ with the given property. Then $\overline{\fq}^f$ and $0^{\prime_f}$
	are the only minimal prime ideals of $\rr$.
	Thus, by \cite[Theorem 2.2]{l}, 
	$\diam (\Gamma (\rr))\leq 2$.
	
	(4) Assume that $\Z (R)$ is not an ideal
	or $J\nsubseteq \Nil (S)$. By Lemma
	\ref{3}, $\Z (\rr)$ is not an ideal.\\
	Let $J\cap \Nil (S) \neq \{0\}$. Then, by
	\cite[Proposition 5.4]{DFF}, $\rr$ is a non-reduced ring.
	Hence, by  \cite[Corollary 2.5]{l},
	$\diam (\Gamma (\rr))=3$.\\
	Let $S$ have more than one 
	 prime ideal  with the given property. Then $\rr$ has 
	 at least two  minimal prime ideal of 
	 the form $\overline{\fq}^f$ and one of the form
	 $\fp^{\prime_f}$. Thus, by
	 \cite[Theorem 2.2 and Corollary 2.5]{l},
	 $\diam (\Gamma (\rr))=3$.\\
	Let $R$  not be a domain.
	By  \cite[Corollary 2.5]{l}, we may assume that
	$\rr$ is reduced, i.e. $R$ is reduced
	 and $J\cap \Nil (S)=\{0\}$. 
	 Then  $R$  has at least two minimal prime ideals, which
	 give rise to two minimal prime ideals in $\rr$.
	 Also, $\rr$  has at least one minimal prime ideal
	 of the form $\overline{\fq}^f$.
	 Thus, by
	 \cite[Theorem 2.2]{l}, $\diam (\Gamma (\rr))=3$.
\end{proof}

\begin{rem}\label{5}
\textnormal{
In the part (2) of the above theorem we separate the
case $\rr \cong \mathbb{Z}_{2}\times \mathbb{Z}_{2}$.	
In fact, when this is the case,  properties (a), (b), and (c)  hold, but $J^2\neq 0$. To see this, let 
$\rr \cong \mathbb{Z}_{2}\times \mathbb{Z}_{2}$. Then
$\rr$ is reduced with exactly two prime ideals.
Thus, by \cite[Proposition 5.4]{DFF}, 
$R$ is reduced and $J\cap \Nil (S)=\{0\}$. Hence
$\rr$ has 
one prime ideal of the form $\overline{\fq}^f$
and one  of the form $\fp^{\prime_f}$. Therefore
$R$ is a field since it is a reduced ring
with only one prime ideal.  Thus properties (a), (b), and (c)  hold.
}
\end{rem}
%%%%%%%%%%%%%%%%%%%%%%%%%%%%%%%%%%%%

We can now deduce the following corollary from
 Theorem \ref{fine}. The corollary shows that one can
easily construct examples of the zero-divisor graph of 
amalgamations of diameter $3$ by
imposing conditions on $R$ only or on $J$ only.
\begin{cor}\label{sade}
	Let $R\bowtie^f J$ 
	have finitely many  elements. Then the following hold:
	\begin{itemize}
		\item [(1)] If $\Z (R)$ is not an ideal, 
		then $\diam (\Gamma (\rr))=3$.
		\item [(2)] If $J\nsubseteq \Nil (S)$ and
		 $J\cap \Nil (S) \neq \{0\}$, 
		 then $\diam (\Gamma (\rr))=3$.
	\end{itemize}
\end{cor}

\begin{cor}
	Let $m,n\in \Bbb N$ with $n|m$ and $n\geq 2$. 
	Assume that $m$ is not a prime number.
	Then $\diam (\Gamma (\mathbb{Z}_{m}\times \mathbb{Z}_{n}))=3$.
\end{cor}
\begin{proof}
   Let $R=\mathbb{Z}_m$, $S=\mathbb{Z}_{m}\times \mathbb{Z}_{n}$, and
   $J$ be the ideal generated by $(0,1)$. Let $f:R\to S$
   be the natural homomorphism. Then
   $f^{-1}(J)=0$,
   and so $\rr \cong f(R)+J=\mathbb{Z}_{m}\times \mathbb{Z}_{n}$. 
   It follows from the comments at the beginning of this section
   that $\Z (\mathbb{Z}_{m})$ is not an ideal. Thus, by 
   Corollary \ref{sade},  we have 
    $\diam (\Gamma (\mathbb{Z}_{m}\times \mathbb{Z}_{n}))=3$. 
\end{proof}

We now make explicit how our results generalize and
extend previous results and examples. In each of these examples,
the homomorphism
$f: R \to S$ is defined naturally.
\begin{exam}
	\begin{itemize}

		\item [(1)] (\cite[Example 3.1]{as06})
		$\diam (\Gamma (\mathbb{Z}_{9} \ltimes \mathbb{Z}_{9}))=2$;
		since $\Z (R)$ 
		is an ideal and $J\subseteq \Nil (S)$.
		\item [(2)] (\cite[Example 3.5]{as06})
		$\diam (\Gamma (\mathbb{Z}_{8} \ltimes \mathbb{Z}_{2}))=2$;
		since $\Z (R)$ 
		is an ideal and $J\subseteq \Nil (S)$.
		\item [(3)] (\cite[Example 3.6]{as06}, \cite[Example 3.8]{a})
		$\diam (\Gamma (\mathbb{Z}_{6} \ltimes \mathbb{Z}_{6}))=3$;
		since $\Z (R)$ is not an ideal.
		\item [(4)] (\cite[Example 3.8]{as06})
		$\diam (\Gamma ((\mathbb{Z}_{6} \times \mathbb{Z}_{6}) \ltimes \mathbb{Z}_{6}))=3$;
		since $\Z (R)$ is not an ideal.
		\item [(5)] (\cite[Example 3.5]{a})	
		Let $R=\mathbb{Z}_{12}$, $S=\mathbb{Z}_{6}\times \mathbb{Z}_{6}$, and  $J$ be any ideal of $S$. Then
		$\diam (\Gamma (R\bowtie^f J))=3$;
		since $\Z (R)$ is not an ideal.
		\item [(6)] (\cite[Example 3.6]{a})		
		Let $R=\mathbb{Z}_{6}$, $S=\mathbb{Z}_{6}[X]/(X^4)$,
		and $J=(x)S$. Then
		$\diam (\Gamma (R\bowtie^f J))=3$;
		since $\Z (R)$ is not an ideal.
		\item [(7)] (\cite[Example 3.7]{a})		
		Let $R=\mathbb{Z}_{6}$, and $I=\{0,2,4\}$. Then
		$\diam (\Gamma (R\bowtie I))=3$;
		since $\Z (R)$ is not an ideal.
		\item [(8)] (\cite[Example 3.9]{a})		
		Let $R=\mathbb{Z}_{9}$ and  $M=\mathbb{Z}_{9}$. 
		Then, by part (1) of the above theorem,
		$\diam (\Gamma (R\bowtie \Z(R)))=1$, and
		$\diam (\Gamma (R\ltimes M))=2$, as before.		
		\item [(9)] (\cite[Example 3.16]{a}) 
		Let $R=\mathbb{Z}_{8}$, $S=\mathbb{Z}_{4}$, and
		$J=\{0,2\}$. Then
		$\diam (\Gamma (R\bowtie^f J))=2$; since $\Z (R)$ 
		is an ideal and $J\subseteq \Nil (S)$.
		\item [(10)] (\cite[Example 4.1]{y})	
		Let $R= \mathbb{Z}_{2} \times \mathbb{Z}_{2}$ 
		and $I= \mathbb{Z}_{2} \times \{0\}$. Then 
		$\diam (\Gamma (R\bowtie I))=3$;
		since $\Z (R)$ is not an ideal.
		\item [(11)] (\cite[Example 4.2]{y})		
		Let $R=\mathbb{Z}_{6}$, and $I=\{0,3\}$. Then
		$\diam (\Gamma (R\bowtie I))=3$;
		since $\Z (R)$ is not an ideal.	
		\item [(12)] (\cite[Example 4.2]{y})		
		Let $R=\mathbb{Z}_{8}$, and $I=\{0,4\}$. Then
		$\diam (\Gamma (R\bowtie I))=2$; since $\Z (R)$ 
		is an ideal and $J\subseteq \Nil (S)$.	
		\item [(13)] (\cite[Example 2.3]{k})		
		Let $R=\mathbb{Z}_{4}$, $S=\mathbb{Z}_{2}\times \mathbb{Z}_{2}[X]/(X^2)$, and $J=(0,x)S$.
		Then, by part (1) of the above theorem,
		$\diam (\Gamma (R\bowtie^f J))=1$.
		\item [(14)] (\cite[Example 3.12]{k})		
		Let $R= \mathbb{Z}_{2} \times \mathbb{Z}_{4}$,
		$S= \mathbb{Z}_{2} \times \mathbb{Z}_{2}$, and
		$J= \{0\} \times  \mathbb{Z}_{2}$. Then
		$\diam (\Gamma (R\bowtie^f J))=3$;
		since $\Z (R)$ is not an ideal.
		\item [(15)] (\cite[Example 3.13]{k})		
		Let $R= \mathbb{Z}_{6} \times \mathbb{Z}_{3}$,
		$S= \mathbb{Z}_{3} \times \mathbb{Z}_{3}$, and
		$J= \{0\} \times  \mathbb{Z}_{3}$. Then
		$\diam (\Gamma (R\bowtie^f J))=3$;
		since $\Z (R)$ is not an ideal.
		\item [(16)] (\cite[Example 3.14]{k})		
		Let $R=\mathbb{Z}_{4}[X]/(X^2)$, 
		$S= \mathbb{Z}_{4} \times \mathbb{Z}_{4}$, 
		$J= \{0\} \times  \mathbb{Z}_{4}$, and
		$f: R \to S$ be the ring
		homomorphism defined by $f(a+bx)=(a,a)$. Then,
		one can deduce from either part (1) or part (2)
		of Corollary \ref{sade} that
		$\diam (\Gamma (R\bowtie^f J))=3$.
		\item [(17)] (\cite[Example 3.15]{k})			
		Let $R=\mathbb{Z}_{6}$, $S=\mathbb{Z}_{2}[X]/(X^3)$, and $J=xS$. Then
		$\diam (\Gamma (R\bowtie^f J))=3$;
		since $\Z (R)$ is not an ideal.
		\end{itemize}
\end{exam}

\end{document}